\documentclass[12pt]{article}
\usepackage{latexsym,amsfonts,amssymb,mathrsfs,float}

\usepackage{amsthm}

\usepackage{amsmath}

\usepackage[dvips]{color}
\usepackage{colordvi,multicol}

\newtheorem{thm}{Theorem}[section]

\newtheorem{lem}[thm]{Lemma}

\newtheorem{rem}[thm]{Remark}

\numberwithin{equation}{section}

\begin{document}
\title{\bf On the length of the longest consecutive switches}
\author{Chen-Xu Hao , Ze-Chun Hu and Ting Ma\thanks{Corresponding author: College of Mathematics, Sichuan University, Chengdu 610065, China\vskip 0cm E-mail address: matingting2008@yeah.net}\\ \\
 {\small College of Mathematics, Sichuan  University,  China}}

\maketitle
\date{}



\vskip 0.5cm \noindent{\bf Abstract}\quad   An unbiased coin is tossed $n$ times independently and sequentially. In this paper, we will study the length of the longest consecutive switches, and  prove that  the limit behaviors are similar to that of the length of the longest head-run.

\smallskip

\noindent {\bf Keywords: }  longest consecutive switches, longest head-run.

\smallskip

\noindent {\bf Mathematics Subject Classification (2010)}\quad 60F15


\section{Introduction}

An unbiased coin with two sides named by ``head" and ``tail" respectively, is tossed $n$ times independently and sequentially. We use 0 to denote ``tail" and 1 to denote ``head".
For simplicity, we assume that all the random variables in the following are defined in a probability space $(\Omega,\mathcal{F},P)$.
Let $\{X_i,i\geq 1\}$ be a sequence of independent and identically distributed random variables with $P\{X_1=0\}=P\{X_1=1\}=\frac{1}{2}$.  Let $S_0=0, S_n=X_1+\cdots+X_n,~n=1,2,\ldots,$ and
\begin{equation}\label{def-INK}
I(N,K)=\max_{0\leq n\leq N-K}(S_{n+K}-S_n),~~~~N\geq K,~N,K\in\mathbb N.
\end{equation}
Denote by $Z_N$ the largest integer for which $I(N,Z_N)=Z_N$. Then $Z_N$ is the length of the longest head-run of pure heads in $N$ Bernoulli trials.

The statistic $Z_N$ has been long studied because it has extensive applications in reliability theory, biology, quality control, pattern recognition, finance,  etc.
 Erd\"{o}s and R\'{e}nyi (1970)   proved  the following result.

\begin{thm}\label{thm-1.1}
Let $0<C_1<1<C_2<\infty$. Then for almost all $\omega\in \Omega$, there exists a finite $N_0=N_0(\omega,C_1,C_2)$ such that
$$
[C_1\log N]\leq Z_N\leq [C_2\log N]
$$
if $N\geq N_0$.
\end{thm}
Hereafter, we denote by ``log" the logarithm with base 2, and by  $\left[x\right]$  the largest integer which is no more than $x$.  Theorem 1.1 was extended by
Koml\'{o}s and Tusn\'{a}dy (1975).

Erd\"{o}s and R\'{e}v\'{e}sz (1976)  presented several sharper bounds of $Z_N$ including the following four theorems among other things.

\begin{thm}\label{thm-1.2}
Let $\varepsilon$ be any positive number. Then for almost all $\omega\in \Omega$, there exists a finite $N_0=N_0(\omega,\varepsilon)$ such that if $N\geq N_0$, then
$$
Z_N\geq[\log N-\log\log\log N+\log\log e-2- \varepsilon].
$$
\end{thm}

\begin{thm}\label{thm-1.3}
Let $\varepsilon$ be any positive number. Then for almost all $\omega\in \Omega$, there exists an infinite sequence $N_i=N_i(\omega,\varepsilon)\  (i=1,2,...)$ of integers such that
$$
Z_{N_i}<[\log N_i-\log\log\log N_i+\log\log e-1+ \varepsilon].
$$
\end{thm}

\begin{thm}\label{thm-1.4}
Let $\{\gamma_n\}$ be a sequence of positive numbers for which
$
\sum_{n=1}^{\infty}2^{- \gamma_n}=\infty.
$
Then for almost all $\omega\in \Omega$, there exists an infinite sequence $N_i=N_i(\omega,\{\gamma_n\}) \  (i=1,2,...)$ of integers such that
$
Z_{N_i}\geq \gamma_{N_i}.
$
\end{thm}

\begin{thm}\label{thm-1.5}
Let $\{\delta_n\}$ be a sequence of positive numbers for which
$
\sum_{n=1}^{\infty}2^{- \delta_n}<\infty.
$
Then for almost all $\omega\in \Omega$, there exists a positive integer $N_0=N_0(\omega,\{\delta_n\})$ such that
$
Z_N< \delta_N
$
if $N\geq N_0$.
\end{thm}

These limit theorems have been extended by many authors.  We refer to Guibas and Odlyzko (1980), Samarova (1981), Kusolitsch and Nemetz (1982), Nemetz and Kusolitsch (1982),  Grill (1987) and   Vaggelatou (2003).

The distribution  function of $Z_N$ and some related problems  have been studied by  Goncharov (1943), F\"{o}ldes (1979), Arratia et al. (1989), Novak (1989, 1991, 1992), Schilling (1990),
Binswanger and Embrechts (1994), Muselli (2000), Vaggelatou (2003),  T\'{u}ri (2009), Novak (2017).  Mao et al. (2015) studied the large deviation behavior for the length of  the longest head run.

Anush (2012)  posed the definition of ``switch", and  considered the bounds for the number of coin tossing switches.  Li (2013) considered the number of switches in unbiased coin-tossing, and
established the central limit theorem and the large deviation principle for the total number of switches.
According to Li (2013), a ``head" switch is the tail followed by a head and a ``tail" switch is the head followed by a tail.

Motivated by the study of the longest head-run  and the work of Li (2013), we will study the length of the longest consecutive switches in this paper. At first, we introduce some notations.

For $m,n\in\mathbb N$, define
\begin{eqnarray*}
&&S_n^{(m)}(H)=\sum_{i=m+1}^{n+m-1}(1-X_{i-1})X_i,\quad\quad S_n^{(m)}(T)=\sum_{i=m+1}^{n+m-1}X_{i-1}(1-X_i).
\end{eqnarray*}
Then $S_n^{(m)}(H)$ (resp. $S_n^{(m)}(T)$) denotes the number of ``head" switches (resp. ``tail" switches) in the trials $\{X_m,X_{m+1}, \ldots,X_{m+n-1}\}$. Set
\begin{equation}\label{def-SNM}
S_n^{(m)}=S_n^{(m)}(H)+S_n^{(m)}(T).
\end{equation}
Then $S_n^{(m)}$ denotes the total number of switches in the sequence $\{X_m,X_{m+1},\ldots,X_{m+n-1}\}$.

For $i,N\in\mathbb N$, define
$$
H_{n,i}^{(N)}=\bigcup_{i\leq m\leq i+N-n+1}\{S_n^{(m)}=n-1\},~~n=1,\ldots, N.
$$
Then $\omega\in H_{n,i}^{(N)}$ implies that there exists at least one sequence of consecutive switches of length $n-1$ in the sequence $\{X_{i}(\omega),X_{i+1}(\omega),\ldots,X_{i+N-1}(\omega)\}$. Define
\begin{equation}\label{NUM-SWITCH}
M_N^{(i)}=\max_{1\leq n\leq N}\left\{n-1|~H_{n,i}^{(N)}\neq \emptyset\right\},
\end{equation}
which stands for the number of switches of the longest consecutive switches in the sequence $\{X_{i},X_{i+1}, \ldots,X_{i+N-1}\}$.  When $i=1$, we denote $M_N$ instead of $M_N^{(1)}$.

 We use $A_1+A_2+ \cdots+A_n$ instead of $A_1\cup A_2\cup \cdots\cup A_n$ when the sets $A_i,i=1,\ldots,n$ are disjoint.


The rest of this paper is organized as follows. In Section 2, we give some illustration about the difference between $M_N$ and $Z_N$.  In Section 3, we present  main  results and some remarks. The proofs will be given in Section 4. In Section 5, we give some final remarks.

\section{The difference of $M_{N}$ and $Z_N$}

In this section, we want to give some illustration about the difference between $M_N$ and $Z_N$.
Suppose that $N\in\mathbb N, N\geq 2$.  Then by the definitions of $M_N$ and $Z_N$, we know that
$$
M_N\in \{0,1,\ldots,N-1\},\ Z_N\in \{0,1,\ldots,N\}.
$$
For $k\in \{1,\ldots,N\}$, define
\begin{equation*}
\label{(2.1)}
A_N^{(k)}=\{M_N=k-1\},\quad\quad B_N^{(k)}=\{Z_N=k\}.
\end{equation*}

Consider a sequence $\{X_i,\ldots,X_{i+k-1}\}$ of length $k$. Without loss of generality, we assume that $k$ is an even number. If it has pure heads, then we have only one case $X_i=\cdots=X_{i+k-1}=1$. If it has consecutive switches of length $k$, then we have the following two cases:
$$
X_i=1,X_{i+1}=0,\cdots, X_{i+k-2}=1,X_{i+k-1}=0,
$$
or
$$
X_i=0,X_{i+1}=1,\cdots, X_{i+k-2}=0,X_{i+k-1}=1.
$$
Based on the above rough analysis, we might conjecture that
\begin{equation*}\label{(2.2)}
P(A_N^{(k)})=2P(B_N^{(k)}).
\end{equation*}
However, it is not true. In the following, we use some concrete examples to illustrate that
the relation between  $M_N$ and $Z_N$ is complex.

\begin{itemize}
 \item Case 1: $N=2$
\begin{table}[H]
\caption{自动换行}
\begin{center}
\begin{tabular}{|p{91pt}|p{60pt}|p{60pt}|}
\hline
$M_{2}$ & $0$ & $1$ \\
\hline
Events & \{00\} \{11\} & \{01\} \{10\}\\
\hline
Number of Events & 2 & 2 \\
\hline
\end{tabular}
\end{center}
\end{table}

\begin{table}[H]
\caption{自动换行}
\begin{center}
\begin{tabular}{|p{91pt}|p{40pt}|p{60pt}|p{40pt}|}
\hline
$Z_{2}$ & $0$ & $1$ & $2$\\ \hline
Events & \{00\}& \{01\} \{10\}&\{11\}\\
\hline
Number of Events & 1 & 2 & 1 \\
\hline
\end{tabular}
\end{center}
\end{table}

\item Case 2: $N=3$
\begin{table}[H]
\caption{自动换行}
\begin{center}
\begin{tabular}{|p{91pt}|p{70pt}|p{130pt}|p{70pt}|}
\hline
$M_{3}$ & $0$ & $1$ &$2$ \\
\hline
Events & \{000\} \{111\}&\{100\}  \{001\}  \{011\}  \{110\}& \{101\} \{010\} \\
\hline
Number of Events & 2 &4& 2 \\
\hline
\end{tabular}
\end{center}
\end{table}

\begin{table}[H]
\caption{自动换行}
\begin{center}
\begin{tabular}{|p{91pt}|p{40pt}|p{130pt}|p{70pt}|p{40pt}|}
\hline
$Z_{3}$ & $0$ & $1$ & $2$ & $3$  \\ \hline
Events&\{000\}& \{100\} \{010\}  \{001\} \{101\}&\{110\}  \{011\}& \{111\}\\
\hline
Number of Events & 1 & 4 & 2 & 1 \\

\hline
\end{tabular}
\end{center}
\end{table}

\newpage

\item Case 3: $N=4$
\begin{table}[H]
\caption{自动换行}
\begin{center}
\begin{tabular}{|p{91pt}|p{40pt}|p{75pt}|p{40pt}|p{40pt}|}
\hline
$M_{4}$ & $0$ & $1$ & $2$ & $3$\\
\hline
Events & \{0000\} \{1111\} & \{1000\} \{0001\} \{0111\} \{1110\} \{0011\} \{1100\} \{0110\} \{1001\} & \{0100\} \{1011\} \{1101\} \{0010\}& \{1010\} \{0101\} \\
\hline
Number of Events & 2 &8&4& 2 \\
\hline
\end{tabular}
\end{center}
\end{table}

\begin{table}[H]
\caption{自动换行}
\begin{center}
\begin{tabular}{|p{91pt}|p{40pt}|p{75pt}|p{75pt}|p{40pt}|p{40pt}|}
\hline
$Z_{4}$ & $0$ & $1$ & $2$ & $3$ & $4$  \\ \hline
Events & \{0000\}& \{1000\} \{0100\} \{0010\} \{0001\} \{1001\} \{1010\} \{0101\} & \{1101\} \{1100\} \{0110\} \{1011\} \{0011\}& \{1110\} \{0111\} & \{1111\} \\
\hline
Number of Events & 1 & 7&5 & 2 & 1 \\

\hline
\end{tabular}
\end{center}
\end{table}

\item Case 4: $N=5$
\begin{table}[H]
\caption{自动换行}
\begin{center}
\begin{tabular}{|p{91pt}|p{40pt}|p{88pt}|p{88pt}|p{40pt}|p{40pt}|}
\hline
$M_{5}$ & $0$ & $1$ & $2$ & $3$ & $4$\\
\hline
Events& \{00000\} \{11111\}& \{10000\} \{00001\} \{01111\} \{11110\} \{00111\} \{11000\} \{00011\} \{11100\} \{10001\} \{10011\}
 \{01110\} \{01100\} \{11001\} ~\{00110\}& \{01000\} \{10111\} \{00100\} \{11011\} \{01101\}
 \{00010\} \{01101\} \{10010\} \{01001\} \{10110\}& \{01011\} \{10100\} \{00101\} \{11010\}& \{10101\} \{01010\}\\
\hline
Number of Events & 2 &14&10&4& 2 \\
\hline
\end{tabular}
\end{center}
\end{table}

\begin{table}[H]
\caption{自动换行}
\begin{center}
\begin{tabular}{|p{91pt}|p{36pt}|p{86pt}|p{86pt}|p{36pt}|p{36pt}|p{37pt}|}
\hline
$Z_{5}$ & $0$ & $1$ & $2$ & $3$ & $4$ & $5$ \\ \hline
Events& \{00000\}  & \{10000\} \{01000\} \{00100\} \{00010\} \{00001\}
\{10100\} \{10010\} \{10001\} \{01010\} \{01001\}
\{00101\} \{10101\}& \{11000\} \{01100\} \{00110\} \{00011\} \{11010\}
\{11001\} \{01101\} \{10110\} \{11011\} \{10011\}
\{01011\}  & \{11100\} \{01110\} \{00111\} \{11101\} \{10111\} & \{01111\} \{11110\}& \{11111\}\\
\hline
Number of Events & 1 & 12&11&5 & 2 & 1 \\
\hline
\end{tabular}
\end{center}
\end{table}
\end{itemize}

From the above concrete examples, we might say that the distributions of $M_N$ and $Z_N$ are very different.
In the rest of this paper, we will show that as $N\to\infty$, their limit behaviors are similar.

\section{Main results and some remarks}\label{MAIN-RES}

In this section, we present several limit results on $M_N$. Corresponding to Theorems \ref{thm-1.1}-\ref{thm-1.5}, we have the following five theorems.

\begin{thm} \label{THM-MN}
We have
\begin{equation} \label{alpha1}
\lim_{N\to\infty}\frac{M_N}{\log N}=1 \quad a.s.
\end{equation}
\end{thm}

\begin{thm}\label{THM01}
Let $\varepsilon$ be any positive number. Then for almost all $\omega\in \Omega$, there exists a finite $N_0=N_0(\omega,\varepsilon)$ such that
\begin{equation} \label{alpha2}
M_N\geq[\log N-\log\log\log N+\log\log e-1- \varepsilon]:=\alpha_1(N)
\end{equation}
if $N\geq N_0$.
\end{thm}

\begin{thm}\label{THM02}
Let $\varepsilon$ be any positive number. Then for almost all $\omega\in \Omega$, there exists an infinite sequence $N_i=N_i(\omega,\varepsilon) (i=1,2,...)$ of integers such that
\begin{equation}\label{alpha02}
M_{N_i}<[\log N_i-\log\log\log N_i+\log\log e+\varepsilon]:=\alpha_2(N_i).
\end{equation}
\end{thm}

\begin{thm}\label{THM03}
Let $\{\gamma_n\}$ be a sequence of positive numbers for which $\sum_{n=1}^{\infty}2^{- \gamma_n}=\infty.$
Then for almost all $\omega\in \Omega$, there exists an infinite sequence $N_i=N_i(\omega,\{\gamma_n\}) \ (i=1,2,...)$ of integers such that
$M_{N_i}\geq \gamma_{N_i}-1.$
\end{thm}

\begin{thm}\label{THM04}
Let $\{\delta_n\}$ be a sequence of positive numbers for which $\sum_{n=1}^{\infty}{2^{-\delta_n}}<\infty.$
Then for almost all $\omega\in \Omega$, there exists a positive integer $N_0=N_0(\omega,\{\delta_n\})$ such that $M_N< \delta_N -1$
if $N\geq N_0$.
\end{thm}

The last two theorems can be reformulated as follows:

\noindent {\bf Theorem 3.4*} {\it
Let $\{\gamma_n\}$ be a sequence of positive numbers for which $\sum_{n=1}^{\infty}2^{- \gamma_n}=\infty.$
Then for almost all $\omega\in \Omega$, there exists an infinite sequence $N_i=N_i(\omega,\{\gamma_n\}) \ (i=1,2,...)$ of integers such that
$S^{(N_i-\gamma_{N_i})}_{N_i}\geq \gamma_{N_i}-1.$}

\noindent {\bf Theorem 3.5*} {\it
Let $\{\delta_n\}$ be a sequence of positive numbers for which $\sum_{n=1}^{\infty}{2^{-\delta_n}}<\infty.$
Then for almost all $\omega\in \Omega$, there exists a positive integer $N_0=N_0(\omega,\{\delta_n\})$ such that $S^{(N-\delta_N)}_N< \delta_N -1$
if $N\geq N_0$.
}

\begin{rem}
(i)  From Theorems 3.1-3.5, we know that $M_N$ and $Z_N$ have similar limiting behaviors.

(ii) The closely related result with respect to Theorems \ref{THM03} and \ref{THM04} is Guibas and Odlyzko (1980, Theorem 1).

\end{rem}

\section{Proofs}
\subsection{Proof of Theorem \ref{THM-MN}}

{\bf Step 1.} We prove
\begin{equation} \label{LEF}
\liminf_{N\to {\infty}}\frac{M_N}{\log N}\geq 1 \quad a.s.
\end{equation}

For any $0<\varepsilon<1$, $N\in\mathbb N$ and $N\geq 2$ , we introduce the following notations:
\begin{eqnarray*}
&&t=\left[(1-\varepsilon)\log N\right]+1,\\
&&\bar{N}=\left[N/t\right]-1,\\
&&U_k=S_{t}^{(tk+1)},\quad k=0,1,\ldots, \bar{N},
\end{eqnarray*}
where $S_{t}^{(tk+1)}$ is defined by (\ref{def-SNM}).  Then the sequence  $\{U_k,0\leq k\leq \bar{N}\}$ of random variables are independent and identically distributed with
\[
U_k\leq t-1 ~~\mbox{and}~~ P\{U_k=t-1\}=2\cdot\frac{1}{2^t}=\frac{1}{2^{t-1}}.
\]
It follows that
\begin{equation*}
P\{U_0<t-1,U_1<t-1,\cdots,U_{\bar{N}}<t-1\}=\left(1-\frac{1}{2^{t-1}}\right)^{\bar{N}+1}.
\end{equation*}

By a simple calculation, we get that
$$
\sum_{N=1}^{\infty}\left(1-\frac{1}{2^{t-1}}\right)^{\bar{N}+1}<\infty.
$$
Then by the Borel-Cantelli lemma, we get that
\begin{equation*}
\liminf_{N\to \infty}\frac{M_N}{\log N}\geq 1-\varepsilon.
\end{equation*}
By the arbitrariness of $\varepsilon$, we obtain (\ref{LEF}).

\noindent {\bf Step 2.} We  prove
\begin{equation}\label{LEF-2}
\limsup_{N\to {\infty}}\frac{M_N}{\log N}\leq 1  \quad a.s.
\end{equation}

For any $\varepsilon>0$ and $N\in\mathbb N$, we introduce the following notations:
\begin{equation*}\label{NOTA-U-AN}
\begin{aligned}
&u=\left[(1+\varepsilon)\log N\right]+1,\\
&A_N=\bigcup\limits_{k=1}^{N-u+1} \{S_{u}^{(k)}=u-1\}. \\
\end{aligned}
\end{equation*}
We have  $P\{S_{u}^{(k)}=u-1\}=\frac{1}{2^{u-1}}$ and thus  $P(A_N)\leq\frac{N}{2^{[(1+\varepsilon)\log N]}}$. For any $T\in\mathbb N$ with $T\varepsilon>1$, $k\in\mathbb N$, it holds that
\begin{equation*}\begin{aligned}
\sum_{k=1}^{\infty}P(A_{k^T})\leq&2\sum_{k=1}^{\infty}\frac{k^T}{k^{T(1+\varepsilon)}}
=2\sum_{k=1}^{\infty}\frac{1}{k^{T\varepsilon}}<{\infty},
\end{aligned}\end{equation*}
which together with the Borel-Cantelli lemma implies that
\begin{equation*}
P(A_{k^{T}} \ i.o.)=P\left(\limsup_{k\to{\infty}} \bigcup\limits_{i=0}^{k^T-u+1}\left\{S_u^{(i)}=u-1\right\}\right)=0.
\end{equation*}
It follows that
\begin{eqnarray}\label{4.3}
\limsup_{k\to\infty}\frac{M_{k^T}}{\log k^T}\leq 1 \quad\quad a.s.
\end{eqnarray}

Let $k^T< n<(k+1)^T$. By (\ref{4.3}), we have
$$
M_n\leq M_{(k+1)^T}\leq (1+\varepsilon)\log (k+1)^T\leq (1+2\varepsilon)\log k^T\leq (1+2\varepsilon)\log n
$$
with probability 1 for all but finitely many $n$. Hence (\ref{LEF-2}) holds.\hfill\fbox

\subsection{Proofs for Theorems 3.2-3.5}

The basic idea comes from \cite{ER76}. For the reader's convenience, we spell out the details. At first, as in \cite[Theorem 5]{ER76},  we give an estimate for the length of consecutive switches, which is very useful in our proofs for Theorems 3.2-3.5.

\begin{thm} \label{THM-EST-SN1} Let $N,K\in\mathbb N$ and let $M_N$ be defined in (\ref{NUM-SWITCH}) with $i=1$. Then
\begin{equation}\label{EST-SN1}
\left(1-\frac{K+2}{2^{K}}\right)^{[\frac{N}{K}]-1}\leq  P\Big(M_N<K-1\Big)\leq\left(1-\frac{K+2}{2^{K}}\right)^{[ \frac{1}{2}[\frac{N}{K}]]}
\end{equation}
if $N\geq 2K$.
\end{thm}

To prove Theorem \ref{THM-EST-SN1}, we need the following lemma.

\begin{lem}\label{LEM-S2N}
Let $N,m\in\mathbb N$ and let $M_N^{(m)}$ be defined in (\ref{NUM-SWITCH}). Then
\begin{equation}
P\left(M_{2N}^{(m)}\geq N-1\right)=\frac{N+2}{2^N}.
\end{equation}
\end{lem}

\begin{proof} Since $\{M_{2N}^{(i)},i\in\mathbb N\}$ are identically distributed, we only consider the case that $i=1$ in the following.
Let
\[A=\{M_{2N}\geq N-1\},~~A_k=\{M_N^{(k+1)}= N-1\},~k=0,1,\ldots,N.\]
Then we have
\begin{equation}
A=A_0+\bar{A_0}A_1+\bar{A_0}\bar{A_1}A_2+\cdots+\bar{A_0}\bar{A_1}\cdots\overline{A_{N-1}}A_N,
\end{equation}
and
\begin{equation*}
P(A_0)=\frac{1}{2^{N-1}},~~~~
P(\bar{A_0}\bar{A_1}\cdots\overline{A_{k-1}}A_k)=\frac{1}{2^N},~k=1,\ldots,N.
\end{equation*}
Hence
\begin{equation*}
\begin{aligned}
P(A)=&P(A_0)+P(\bar{A_0}A_1)+P(\bar{A_0}\bar{A_1}A_2)+\cdots+P(\bar{A_0}\bar{A_1}\cdots\overline{A_{N-1}}A_N)\\
=&\frac{1}{2^{N-1}}+\frac{N}{2^{N}}
=\frac{N+2}{2^N}.
\end{aligned}
\end{equation*}
\end{proof}

\begin{proof}[{\bf Proof of Theorem \ref{THM-EST-SN1} }]
Let $N,K\in\mathbb N$ with $N\geq 2K$. Denote
\begin{eqnarray*}
&&B_j=\{M_K^{(j+1)}\geq K-1\},\  j=0,1,\ldots,N-K,\\
&&C_l=\bigcup_{j=lK}^{(l+1)K}B_j,\ l=0,1,\ldots,\left[\frac{N-2K}{K}\right].
\end{eqnarray*}
Then for any $l=0,1,\ldots,[\frac{N-2K}{K}]$, we have $C_l=\{M_{2K}^{(lK+1)}\geq K-1\},$
and for any $l=0,1,\ldots,[\frac{N-2K}{K}]-2$, we have $C_l\cap C_{l+2}=\emptyset$.
By Lemma \ref{LEM-S2N}, we know that for any $l=0,1,\ldots,[\frac{N-2K}{K}]$,
$$P(C_l)=\frac{K+2}{2^K}.$$

Let
\begin{equation}\label{D_01}
\begin{aligned}
D_0=&C_0+C_2+\ldots+C_{2[\frac{1}{2}[\frac{N-2K}{K}]]},\\
D_1=&C_1+C_3+\ldots+C_{2[\frac{1}{2}([\frac{N-2K}{K}]-1)]+1}.
\end{aligned}
\end{equation}

By the independence of $\{C_0,C_2,\ldots,C_{2[\frac{1}{2}[\frac{N-2K}{K}]]}\}$, we have
\begin{equation}\label{D0}
P(\bar{D_0})=P(\bar{C_0})P(\bar{C_2})\ldots P(\overline{C_{2[\frac{1}{2}[\frac{N-2K}{K}]]}})
=\left(1-\frac{K+2}{2^K}\right)^{[\frac{1}{2}[\frac{N-2K}{K}]]+1}.
\end{equation}
Similarly we have
\begin{equation}\label{D1}
P(\bar{D_1})=\left(1-\frac{K+2}{2^K}\right)^{[\frac{1}{2}([\frac{N-2K}{K}]-1)]+1}.
\end{equation}

By the obvious fact that $D_0\subset \{M_N\geq K-1\}$,  we get that
\begin{equation}\label{111}
\begin{aligned}
P(M_N<K-1)&\leq P(\bar{D_0})=\left(1-\frac{K+2}{2^K}\right)^{[\frac{1}{2}[\frac{N-2K}{K}]]+1}.
\end{aligned}
\end{equation}

In the following we prove that
\begin{equation*}
P(D_1D_0)\geq P(D_1)P(D_0).
\end{equation*}
To this end,  by (\ref{D_01}), it is enough  to prove that for any $i=2l$, $l=0,1\cdots [\frac{1}{2}[\frac{N-2K}{K}]]$,
\begin{equation*}
P(D_1C_i)\geq P(D_1)P(C_i).
\end{equation*}
Below we give the proof for  $l=0$ and the proofs for $l=1,\cdots,[\frac{1}{2}[\frac{N-2K}{K}]]$ are similar. We omit them.

For $i=1,\ldots,K+1$, denote
\begin{equation*}
\begin{aligned}
F_i=&\{(X_i,\cdots,X_{i+K-1})\ \mbox{is the first section of consecutive switches of length}\ K-1\\
&\quad\quad\quad\quad\quad\quad\quad\quad\   \mbox{in the sequence}\ (X_1,\cdots,X_{2K})\},
\end{aligned}
\end{equation*}
Then we have
\begin{eqnarray*}
F_i\cap F_j=\emptyset,~\forall i\ne j;~C_0=\bigcup_{i=1}^{K+1}F_i,
\end{eqnarray*}
and
\begin{equation*}
\begin{aligned}
P(F_1)=&P\{(X_1\cdots X_K)\ \mbox{has consecutive switches}\}=\frac{1}{2^{K-1}},\\
P(F_i)=&P\{(X_j,\cdots,X_{j+K-1})\ \mbox{has consecutive switches}, X_{j-1}=X_j\}=\frac{1}{2^K},~i=2,\ldots, K.
\end{aligned}
\end{equation*}

By the independence of $\{X_j,~j=1,2\ldots, N\}$, we have
\begin{equation}\label{F1}
P(D_1F_1)=P(D_1)P(F_1),
\end{equation}
\begin{equation}\label{F2}
\begin{aligned}
&P(D_1F_2)\\
=&P\Big(D_1\cap \{(X_2, \cdots, X_{K+1})\text{~has~consecutive~switches},~X_1=X_2\}\Big) \\
=&P\Big(D_1\cap \{(X_2,\cdots,X_{K+1})\text{~has~consecutive~switches},~X_1=X_2,X_{K+1}=1\}\Big)\\
&+P\Big(D_1\cap \{(X_2,\cdots,X_{K+1})\text{~has~consecutive~switches},~X_1=X_2,X_{K+1}=0\}\}\Big)\\
=&2P\Big(D_1\cap \{(X_2,\cdots,X_{K+1})\text{~has~consecutive~switches},~X_1=X_2,X_{K+1}=1\}\Big)\\
=&2P\Big(D_1\cap\{X_{K+1}=1\}\Big)P\Big( \{(X_2,\cdots,X_K)\text{~has~consecutive~switches},~X_K=0,~X_1=X_2\}\Big)\\
=&\frac{1}{2^K}P(D_1)=P(D_1)P(F_2),
\end{aligned}
\end{equation}
\begin{equation}\label{F3-0}
\begin{aligned}
&P(D_1F_3)\quad\quad  (\mbox{suppose that}\ K\geq 3)\\
=&P\Big(D_1\cap\{(X_3,\cdots,X_{K+2})\text{~has~consecutive~switches},~X_2=X_3\}\Big)\\
=&P\Big(D_1\cap\{(X_3,\cdots,X_{K+2})\text{~has~consecutive~switches},~X_2=X_3,~(X_{K+1},X_{K+2})=(0,1)\}\Big)\\
&+P\Big(D_1\cap\{(X_3,\cdots,X_{K+2})\text{~has~consecutive~switches},~X_2=X_3,~(X_{K+1},X_{K+2})=(1,0)\}\Big)\\
=&2P\Big(D_1\cap\{(X_3,\cdots,X_{K+2})\text{~has~consecutive~switches},~X_2=X_3,~(X_{K+1},X_{K+2})=(0,1)\}\Big)\\
=&2P\Big(\{(X_3,\cdots,X_{K})\text{~has~consecutive~switches},~X_2=X_3,X_K=1\}\Big)\\
&\times P\Big(D_1\cap\{(X_{K+1},X_{K+2})=(0,1)\}\Big)\\
=&\frac{1}{2^{K-2}}P\Big(D_1\cap\{(X_{K+1},X_{K+2})=(0,1)\}\Big)=4P\Big(D_1\cap\{(X_{K+1},X_{K+2})=(0,1)\}\Big) P(F_3).
\end{aligned}
\end{equation}
By the definition of $D_1$, we know that
\begin{eqnarray*}
&&P\Big(D_1\cap\{(X_{K+1},X_{K+2})=(0,1)\}\Big)\geq P\Big(D_1\cap\{(X_{K+1},X_{K+2})=(0,0)\}\Big),\\
&&P\Big(D_1\cap\{(X_{K+1},X_{K+2})=(0,1)\}\Big)\geq P\Big(D_1\cap\{(X_{K+1},X_{K+2})=(1,1)\}\Big),\\
&&P\Big(D_1\cap\{(X_{K+1},X_{K+2})=(0,1)\}\Big)=P\Big(D_1\cap\{(X_{K+1},X_{K+2})=(1,0)\}\Big),
\end{eqnarray*}
which together with (\ref{F3-0}) implies that
\begin{equation}\label{F3}
\begin{aligned}
&P(D_1F_3)-P(D_1)P(F_3)\\
=&\Big\{P\big(\{(X_{K+1},X_{K+2})=(0,1)\}\cap D_1\big)-P\big(\{(X_{K+1},X_{K+2})=(0,0)\}\cap D_1\big)\\
&+P\big(\{(X_{K+1},X_{K+2})=(0,1)\}\cap D_1\big)-P\big(\{(X_{K+1},X_{K+2})=(1,1)\}\cap D_1\big)\Big\}P(F_3)\\
\geq& 0.
\end{aligned}
\end{equation}
Similarly, if $K\geq 4$,  we have  that
\begin{equation}\label{F4}
    P(D_1F_i)\geq P(D_1)P(F_i), \forall  i=4,\ldots, K.
\end{equation}
Finally, by the definitions of $D_1$ and $F_{K+1}$, we know that $D_1\cap F_{K+1}=F_{K+1}$. Hence we have
\begin{equation}\label{F5}
  P(D_1F_{K+1})=P(F_{K+1})\geq P(D_1)P(F_{K+1}).
\end{equation}

By \eqref{F1},  \eqref{F2},  \eqref{F3},  \eqref{F4} and \eqref{F5}, we obtain
\begin{equation*}
P(D_1\cap C_0)=\sum_{i=1}^{K+1}P(D_1\cap F_i)\geq\sum_{i=1}^{K+1}P(F_i)P(D_1)=P(D_1)P(C_0).
\end{equation*}

It is easy to check that
\begin{equation*}
P(D_1D_0)\geq P(D_1)P(D_0)\Leftrightarrow P(\bar{D_1}\bar{D_0})\geq P(\bar{D_1})P(\bar{D_0}).
\end{equation*}
Then by (\ref{D0}) and (\ref{D1}), we get
\begin{equation}
P(\bar{D_0}\bar{D_1})\geq\Big(1-\frac{K+2}{2^K}\Big)^{[\frac{1}{2}[\frac{N-2K}{K}]]+1+[\frac{1}{2}([\frac{N-2K}{K}]-1)]+1}.\label{121}
\end{equation}
As to the right-hand side of  (\ref{121}), we have
\begin{itemize}
  \item [(i)] when $[\frac{N-2K}{K}]$ is even, the exponential part on the right-hand side is equal to
  \[
  \frac{1}{2}\Big[\frac{N-2K}{K}\Big]+\frac{1}{2}\Big[\frac{N-2K}{K}\Big]-1+2=\Big[\frac{N-2K}{K}\Big]+1;
  \]
   \item [(ii)] when  $[\frac{N-2K}{K}]$ is odd, the exponential part on the right-hand is equal to
   \[
   \frac{1}{2}\Big[\frac{N-2K}{K}\Big]-\frac{1}{2}+\Big[\frac{1}{2}\Big[\frac{N-2K}{K}\Big]-\frac{1}{2}\Big]+2=\Big[\frac{N-2K}{K}\Big]+1.
   \]
\end{itemize}
Hence
\begin{equation}\label{SN1-GEQ}
P(M_N<K-1)=P(\bar{D_0}\bar{D_1})\geq \Big(1-\frac{K+2}{2^K}\Big)^{[\frac{N-2K}{K}]+1}.
\end{equation}
By ($\ref{111}$) and (\ref{SN1-GEQ}), we complete the proof.
\end{proof}

To prove Theorem $\ref{THM01}$, we need the following lemma.
\begin{lem}\label{Lem01}
Let $\{\alpha_j,~j\geq 1\}$ be a sequence of positive numbers. Suppose that $\lim\limits_{j\to \infty}\alpha_j=a>0$. Then we have
\begin{equation*}
\sum_{j=1}^{\infty}\frac{1}{\alpha_j^{\log j}}<+\infty~~\text{if $a>2$},
\end{equation*}
and
\begin{equation*}
\sum_{j=1}^{\infty}\frac{1}{\alpha_j^{\log j}}=+\infty~~\text{if $a\leq2$}.
\end{equation*}
\end{lem}
\begin{proof}
\begin{equation*}
\sum_{j=1}^{\infty}\frac{1}{\alpha_j^{\log j}}=\sum_{j=1}^{\infty}\frac{1}{\alpha_j^{\frac{\log_{\alpha_j}j}{\log_{\alpha_j}2}}}=\sum_{j=1}^{\infty}j^{-\frac{1}{\log_{\alpha_j}2}}=\sum_{j=1}^{\infty}j^{-\log \alpha_j}.
\end{equation*}
If $a>2$, then $p=\log a>1$ and thus
\begin{equation*}
\sum_{j=1}^{+\infty}j^{-\log\alpha_j}=\sum_{j=1}^{+\infty}j^{-\log a}\frac{j^{-\log\alpha_j}}{j^{-\log a}}=\sum_{j=1}^{+\infty}j^{-p}\frac{j^{-\log\alpha_j}}{j^{-\log a}}<\infty.
\end{equation*}
If $a\leq 2$, then $p=\log a\leq 1$ and thus
\begin{equation*}
\sum_{j=1}^{+\infty}j^{-\log\alpha_j}=\sum_{j=1}^{+\infty}j^{-\log a}\frac{j^{-\log\alpha_j}}{j^{-\log a}}=\sum_{j=1}^{+\infty}j^{-p}\frac{j^{-\log\alpha_j}}{j^{-\log a}}=\infty.
\end{equation*}
\end{proof}

\begin{proof} [{\bf Proof of Theorem \ref{THM01}}]
Let $N_j$ be the smallest integer with $\alpha_1(N_j)+1=j$. Since
$\lim\limits_{j\to\infty}(1-\frac{j+2}{2^j})^{-3/2}=1$, there exists $M\in \mathbb{N}$ such that $M>2$ and
$$
\left(1-\frac{j+2}{2^j}\right)^{-3/2}\leq 2,\ \forall j> M.
$$
Then by Theorem 4.1, we have
\begin{equation*}
\begin{aligned}
\sum_{j=1}^{\infty}P\left\{M_{N_j}<\alpha_1(N_j)\right\}&
\leq\sum_{j=1}^{\infty} \left(1-\frac{j+2}{2^j}\right)^{[\frac{1}{2}[\frac{N_j}{j}]]}\\
&=\sum_{j=1}^{M} \left(1-\frac{j+2}{2^j}\right)^{[\frac{1}{2}[\frac{N_j}{j}]]}+\sum_{j=M+1}^{\infty} \left(1-\frac{j+2}{2^j}\right)^{[\frac{1}{2}[\frac{N_j}{j}]]}\\
&:=\beta+\sum_{j=M+1}^{\infty} \left(1-\frac{j+2}{2^j}\right)^{[\frac{1}{2}[\frac{N_j}{j}]]}\\
&\leq \beta+\sum_{j=M+1}^{\infty} \left(1-\frac{j+2}{2^j}\right)^{\frac{N_j}{2j}-\frac{3}{2}}\\
&\leq \beta+2\sum_{j=M+1}^{\infty}\Big(1-\frac{j+2}{2^j}\Big)^{\frac{N_j}{2j}}\\
&=\beta+2\sum_{j=M+1}^{\infty} (\sqrt{e_j})^{-\frac{N_j}{2^j}\cdot\frac{j+2}{j}},
\end{aligned}
\end{equation*}
where $e_j:=(1-\frac{j+2}{2^j})^{-\frac{2^j}{j+2}}$. By $\lim\limits_{x\to \infty}(1+\frac{1}{x})^x=e$, we have
 \begin{eqnarray}\label{**}
 \lim_{j\to \infty}e_j=e.
 \end{eqnarray}
Without loss of generality, we assume that $e_j\geq 2$ for any $j>M$.

By $\alpha_1(N_j)=j$, we have $$j\leq \log N_j-\log\log\log N_j+\log\log e-1-\varepsilon,$$
and thus
$$\frac{N_j}{2^j}\geq\frac{2^{1+\varepsilon}\cdot\log\log N_j }{\log e},~~\log j<\log\log N_j.$$
Thus
\begin{eqnarray*}
\begin{aligned}
 \sum_{j=M+1}^{\infty} (\sqrt{e_j})^{-\frac{N_j}{2^j}\cdot\frac{j+2}{j}}&
 =  \sum_{j=M+1}^{\infty} \left(\frac{1}{e_j}\right)^{\frac{1}{2}\cdot \frac{N_j}{2^j}\cdot\frac{j+2}{j}}
\leq \sum_{j=M+1}^{\infty} \left(\frac{1}{e_j}\right)^{\frac{1}{2}\cdot \frac{2^{1+\varepsilon}\cdot\log\log N_j }{\log e}\cdot\frac{j+2}{j}}\\
& =\sum_{j=M+1}^{\infty} \Big(\frac{1}{e_j^{\frac{2^{\varepsilon}}{\log e}}}\Big)^{\frac{j+2}{j}\cdot\log\log N_j}\\
& \leq \sum_{j=M+1}^{\infty}\Big(\frac{1}{e_j^{\frac{2^{\varepsilon}}{\log e}}}\Big)^{\frac{j+2}{j}\cdot\log j}.
\end{aligned}
\end{eqnarray*}
For any $\varepsilon>0$, by (\ref{**}), we have
 $$\lim_{j\to\infty}e_j^{\frac{2^{\varepsilon}}{\log e}}=e^{\frac{2^{\varepsilon}}{\log e}}=\left(e^{\ln 2}\right)^{2^{\varepsilon}}=2^{2^{\varepsilon}}>2.$$
Then by Lemma $\ref{Lem01}$ and the Borel-Cantelli lemma, we complete the proof.
\end{proof}

To prove Theorem $\ref{THM02}$, we need the following version of Borel-Cantelli lemma.

\begin{lem}(\cite{KS64})\label{Borel Cantelli}
Let $A_1,A_2,\ldots$ be arbitrary events, fulfilling the conditions
$
\sum_{n=1}^{\infty}P(A_n)=\infty
$
and
\begin{eqnarray}\label{condition2}
\liminf_{n\to \infty}\frac{\sum_{1\leq k<l\leq n}P(A_kA_l)}{\sum_{1\leq k<l\leq n}P(A_k)P(A_l)}=1.
\end{eqnarray}
Then  $P\left(\limsup\limits_{n\to\infty}A_n\right)=1$.
\end{lem}

\begin{proof}[{\bf Proof of Theorem \ref{THM02}}]
Let $\delta>0$. Let $N_j=N_j(\delta)$ be the smallest integer for which $\alpha_2(N_j)=[j^{1+\delta}]$ with $\alpha_2(N_j)$ given by $(\ref{alpha02})$.
Let
\begin{equation}\label{A_j}
A_j=\{M_{N_j}<\alpha_2(N_j)\},~~j\geq1.
\end{equation}
 By Theorem \ref{THM-EST-SN1}, we have
\begin{eqnarray}\label{proof-thm3.3-1}
\sum_{j=1}^{\infty}P(A_j)
&\geq&\sum_{j=1}^{\infty}\left(1-\frac{[j^{1+\delta}]+3}{2^{[j^{1+\delta}]+1}}\right)^{\big[\frac{N_j}{[j^{1+\delta}]+1}\big]-1}\nonumber\\
&\geq&\sum_{j=1}^{\infty}\left(\left(1-\frac{[j^{1+\delta}]+3}{2^{[j^{1+\delta}]+1}}\right)^{-\frac{2^{[j^{1+\delta}]+1}}{[j^{1+\delta}]+3}}\right)^{-\frac{[j^{1+\delta}]+3}
{2^{[j^{1+\delta}]+1}}\cdot{\frac{N_j}{[j^{1+\delta}]+1}}}\nonumber\\
&:=&\sum_{j=1}^{\infty}f_j^{-\frac{N_j}{2^{[j^{1+\delta}]+1}}\cdot\frac{[j^{1+\delta}]+3}{[j^{1+\delta}]+1}},
\end{eqnarray}
 where $f_j=\left(1-\frac{[j^{1+\delta}]+3}{2^{[j^{1+\delta}]+1}}\right)^{-\frac{2^{[j^{1+\delta}]+1}}{[j^{1+\delta}]+3}}$.  As in (\ref{**}), we have
 \begin{equation}\label{proof-thm3.3-2}
 \lim_{j\to\infty}f_j=e.
 \end{equation}
Then there exists $M\in \mathbb{N}$ such that $M>2$ and
\begin{eqnarray}\label{proof-thm3.3-3}
f_j\geq 2,\ \forall j> M.
\end{eqnarray}
 By $(\ref{alpha02})$ we have$$\log N_j-\log\log\log N_j+\log\log e+\varepsilon\leq[j^{1+\delta}]+1,$$
which implies that
$$\frac{N_j}{2^{[j^{1+\delta}]+1}}\leq \frac{\log\log N_j}{2^{\varepsilon}\log e}. $$
Then by (\ref{proof-thm3.3-1}) and (\ref{proof-thm3.3-3}), we get
\begin{eqnarray}\label{proof-thm3.3-4}
\sum_{j=1}^{\infty}P(A_j)
\geq\sum_{j=1}^{M}P(A_j)+\sum_{j=M+1}^{\infty}f_j^{-\frac{
[j^{1+\delta}]+3}{[j^{1+\delta}]+1}\cdot\frac{\log\log N_j}{2^{\varepsilon}\log e}}.
\end{eqnarray}
Let $0<\varepsilon_0<1$ satisfy
$$
[j^{1+\delta}]=[\log N_j-\log\log\log N_j+\log\log e+\varepsilon]>\varepsilon_0\log N_j, \forall j=1,2,\ldots.
$$
Then
$$
\log N_j<\frac{1}{\varepsilon_0}\cdot j^{1+\delta},
$$
which implies that
\begin{eqnarray}\label{proof-thm3.3-5}
\log\log N_j<\log \frac{1}{\varepsilon_0}+(1+\delta)\log j.
\end{eqnarray}
Hence we have
\begin{eqnarray} \label{proof-thm3.3-6}
\sum_{j=M+1}^{\infty}f_j^{-\frac{
[j^{1+\delta}]+3}{[j^{1+\delta}]+1}\cdot\frac{\log\log N_j}{2^{\varepsilon}\log e}}
&=&
\sum_{j=M+1}^{\infty}\left(f_j^{\frac{
                  [j^{1+\delta}]+3}{[j^{1+\delta}]+1}\cdot2^{-\log\log e-\varepsilon}}\right)^{-\log\log N_j}\nonumber\\
&\geq&\sum_{j=M+1}^{\infty}\left(f_j^{\frac{
[j^{1+\delta}]+3}{[j^{1+\delta}]+1}\cdot 2^{-\log\log e-\varepsilon}}\right)
           ^{-\log \frac{1}{\varepsilon_0}-(1+\delta)\log j}.
\end{eqnarray}
For any given positive number $\varepsilon$, by (\ref{proof-thm3.3-2}), we have
\begin{eqnarray}\label{proof-thm3.3-7}
\lim_{j\to\infty}f_j^{\frac{
[j^{1+\delta}]+3}{[j^{1+\delta}]+1}\cdot\frac{1+\delta}{2^{\varepsilon}\log e}}=2^{\frac{1+\delta}{2^{\varepsilon}}}\leq 2,
\end{eqnarray}
when $\delta$ is small enough.  By Lemma \ref{Lem01}, (\ref{proof-thm3.3-4}),  (\ref{proof-thm3.3-6}) and (\ref{proof-thm3.3-7}), we get that $\sum_{n=1}^{\infty}P(A_n)=\infty$.

Recall that $\{A_j,~j\geq1\}$ is defined in (\ref{A_j}). For $i<j$, we define
\begin{eqnarray*}
&&B_{i,j}=
\left\{
\begin{array}{cl}
\{M_{N_i}<\alpha_2(N_j)\},&\ \mbox{if}\ N_i\geq \alpha_2(N_j),\\
\Omega,&\ \mbox{otherwise},
\end{array}
\right.\\
&&C_{i,j}=\left\{M_{N_j-N_i}^{(N_i)}<\alpha_2(N_j)\right\}.
\end{eqnarray*}
We claim that
\begin{eqnarray}\label{proof-thm3.3-8}
P(A_j)=P(B_{i,j})P(C_{i,j})(1+o(1))\ \ \mbox{as}\ j\to\infty.
\end{eqnarray}
In fact, by the definitions of $A_j, B_{i,j}$ and $C_{i,j}$, we know that
\begin{eqnarray}\label{proof-thm3.3-9}
A_j=B_{i,j}\cap C_{i,j}\cap \left\{M^{(N_i-\alpha_2(N_j)+1)}_{2\alpha_2(N_j)}<\alpha_2(N_j)-1\right\}.
\end{eqnarray}
By Lemma \ref{LEM-S2N}, we have
\begin{eqnarray}\label{proof-thm3.3-10}
P\left(M^{(N_i-\alpha_2(N_j)+1)}_{2\alpha_2(N_j)}<\alpha_2(N_j)-1\right)
&\geq &P\left(M^{(N_i-\alpha_2(N_j)+1)}_{2\alpha_2(N_j)}<\alpha_2(N_j)-1\right)\nonumber\\
&=&1-\frac{\alpha_2(N_j)+2}{2^{\alpha_2(N_j)}}\to 1\ \mbox{as}\ j\to \infty.
\end{eqnarray}
 (\ref{proof-thm3.3-9}) and (\ref{proof-thm3.3-10}) imply (\ref{proof-thm3.3-8}).  Similar to (\ref{proof-thm3.3-8}), we have
\begin{eqnarray}\label{proof-thm3.3-11}
P(A_iA_j)=P(A_i)P(C_{i,j})(1+o(1))\ \mbox{as}\ j\to\infty.
\end{eqnarray}
By  (\ref{proof-thm3.3-8}) and (\ref{proof-thm3.3-11}), we get
\begin{eqnarray*}
P(A_iA_j)=\frac{P(A_i)P(A_j)}{P(B_{i,j})}(1+o(1))\ \mbox{as}\ j\to\infty, \end{eqnarray*}
which implies that
\begin{eqnarray}\label{proof-thm3.3-12}
\frac{P(A_iA_j)}{P(A_i)P(A_j)}=\frac{1+o(1)}{P(B_{i,j})}\geq 1+o(1)\ \mbox{as}\ j\to\infty.
\end{eqnarray}
By the fact that $\sum_{n=1}^{\infty}P(A_n)=\infty$ and    (\ref{proof-thm3.3-12}) , we know that
 ($\ref{condition2}$) holds.  By Lemma $\ref{Borel Cantelli}$, we complete the proof.
\end{proof}

\begin{proof}[{\bf Proof of Theorem 3.4*}]

Let $A_n=\{S^{(n-\gamma_n)}_n\geq \gamma_n-1\}$. Then we have
$$
P(A_n)=2\cdot 2^{-\gamma_n},
$$
which together with the assumption implies that
$$
\sum_{n=1}^{\infty}P(A_n)=\infty.
$$
 By following the method in the proof of Theorem \ref{THM02}, we have
 \begin{eqnarray*}
\frac{P(A_iA_j)}{P(A_i)P(A_j)}\geq 1+o(1)\ \mbox{as}\ j\to\infty.
\end{eqnarray*}
 Then we get that ($\ref{condition2}$) holds and thus by Lemma \ref{Borel Cantelli} we complete the proof.
\end{proof}

\begin{proof}[{\bf Proof of Theorem 3.5*}]
Let $B_n=\{S^{(n-\delta_n)}_n\geq  \delta_n-1\}$. Then we have
$$
P(B_n)=2\cdot 2^{-\delta_n},
$$
which together with the assumption implies that
$$
\sum_{n=1}^{\infty}P(B_n)<\infty.
$$
By the Borel-Cantelli lemma, we get the result.
\end{proof}

\section{Final remarks}

After the first version of our paper was uploaded to arXiv, Professor Laurent Tournier sent two emails to us and gave some helpful comments. In particular, he told us one way to reduce sonsecutive switches to pure heads or pure tails by doing the following: introduce a sequence $(Y_n)$ such that
$$
Y_{2n}=X_{2n},\ Y_{2n+1}=1-X_{2n+1}.
$$
Then $(Y_n)$ is again a sequence of independent and unbiased coin tosses.  And a sequence of  consecutive switches for $X$ is equivalent to a sequence of pure heads or pure tails for $Y$. Then Theorems 3.1, 3.2 and 3.5 can be deduced easily from Theorems 1.1, 1.2 and 1.5, respectively.

We spell out all the proofs with two reasons. One is for the reader's convenience.
The other is that as to biased coin tosses, it seems that Theorems 3.1, 3.2 and 3.5 can not be deduced directly from Theorems 1.1, 1.2 and 1.5, respectively, and our proof may be moved to this case. We will consider the biased coin tosses in a forthcoming paper.

\bigskip

\noindent {\bf\large Acknowledgments} \quad   We thank Professor Laurent Tournier for his helpful comments, which improved the first version of this paper. This work was
supported by National Natural Science Foundation of China (Grant
No. 11771309 and  11871184) and  the Fundamental Research Funds for the Central Universities of China.


\begin{thebibliography}{1234}
\bibitem{An12} Anush, 2012. Bounds for number of coin toss switches, http://mathoverflow.net
/questions/116452.

\bibitem{AGG89} Arratia, R., Goldstein, L., Gordon, L., 1989. Two moments suffice for Poisson approximation. Ann. Probab. 17, 9-25.

\bibitem{BE94} Binwanger, K., Embrechts, P., 1994. Longest runs in coin tossing. Insurance: Mathematics and Economics 15,  139-149.


\bibitem{ER70} Erd\"{o}s, P., R\'{e}nyi, A., 1970. On a new law of large numbers. J. Anal. Math. 22, 103-111.

\bibitem{ER76} Erd\"{o}s, P., R\'{e}v\'{e}sz, P., 1976. On the length of the longest head-run. Coll. Math. Soc. J. Bolyai: Topic in Information Theory (ed. Csisz\'{a}r, I. - Elias, P.) 16, 219-228.


\bibitem{Fo79} F\"{o}ldes, A., 1979. The limit distribution of the length of the longest head-run, Periodica Mathematica Hungarica, 10(4),  301-310.

\bibitem{Go43} Goncharov, V.L., 1943. On the field of combinatory analysis. Amer. Math. Soc. Transl. 19, 1-46.

\bibitem{Gr87} Grill, K., 1987. Erd\"{o}s-R\'{e}v\'{e}sz type bounds for the length of the longest run from a stationary mixing sequence. Probab. Theory Relat. Fields 75, 77-85.

\bibitem{GO80} Guibas L.J., Odlyzko, A.M., 1980. Long repetitive patterns in random sequenes, Z. Wahrsch. Verw. Gebiete, 53, 241-262.

\bibitem{KS64} Kochen S., Stone C., 1964. A note on the Borel-Cantelli lemma. Ill. J. Math. 8: 248-251.

\bibitem{KT75} Koml\'{o}s J., Tusn\'{a}dy G., 1975. On sequences of ``pure heads". Ann. Probab. 3(4), 608-617.

\bibitem{KN82} Kusolitsch N., Nemetz, T., 1982. Erd\"{o}s-R\'{e}v\'{e}sz asymptotic bounds and the length of longest paths in certain random graphs. Colloq. Math. Soc. J\'{a}nos Bolyasi 36. Limit theorems in
Probability and Statistics, Veszpr\'{e}m, Hungary, pp. 649-665.

\bibitem{Li13} Li, W., 2013. On the number of switches in unbiased coin-tossing. Statistics Probab. Letters 83, 1613-1618.

\bibitem{MWW15} Mao Y.-H., Wang F., Wu X.-Y., 2015. Large deviation behavior for the longest head run in an IID Bernoulli sequence. J. Theor. Probab. 28, 259-268.

\bibitem{Mu00} Muselli M., 2000. Useful inequalities for the longest run distribution. Statistics Probab. Letters 46, 239-249.

\bibitem{NK82} Nemetz, T., Kusolitsch N.,  1982. On the longest run of coincidences. Z. Wahrsch. Verw. Geb. 61, 59-73.

\bibitem{No89} Novak, S. Y., 1989. Asymptotic expansions in the problem of the length of the longest head-run for  Markov chain with two states. Trudy Inst. Math. 13, 136-147 (in Russian).

\bibitem{No91} Novak, S. Y., 1991.   Rate of convergence in the problem of the longest head run. Siberian Math. J. 32(3), 444-448.

\bibitem{No92} Novak, S. Y., 1992. Longest runs in a sequence of $m$-dependent random variables. Probab. Theory Related Fields 91, 269-281.

\bibitem{No17} Novak, S. Y., 2017. On the length of the longest head run. Statistics Probab. Letters 130, 111-114.



\bibitem{Sa81} Samarova S.S., 1981. On the length of the longest heas-run for a Markov chain with two states. Theory of Probab. and Its Appl. 26(3), 498-509.

\bibitem{Sc90} Schilling M.F., 1990.  The longest run of heads. The College Mathematics J. 21(3),  196-207.

\bibitem{Tu0} T\'{u}ri J., 2009. Limit theorems for the longest run. Annales Mathematicae et Informaticae. 36, 133-141.

\bibitem{Va03} Vaggelatou E., 2003. On the length of the longest run in a multi-state Markov chain. Statistics Probab. Letters 62, 211-221.

\end{thebibliography}
\end{document}